\documentclass[12pt,a4paper]{amsart}
\usepackage{fullpage}
\usepackage{mathrsfs}
\usepackage{CJK}
\usepackage{geometry}
\usepackage{setspace}
\usepackage{amsfonts}
\usepackage{cite}
\usepackage{amssymb}
\usepackage{color, latexsym}
\usepackage{graphicx}
\usepackage{amsbsy,amscd}
\usepackage{fancyhdr}
\usepackage[colorlinks=true,
           citecolor=blue,
           linkcolor =red]{hyperref}

\allowdisplaybreaks

\fontsize{9pt}{9pt}\selectfont
\def\bl{\boldsymbol{\ell}}
\def\bmu{\boldsymbol{\mu}}
\def\bgl{\boldsymbol{\lambda}}

\def\bc{\mathbf{c}}
\def\bi{\mathbf{i}}
\def\bk{\mathbf{k}}
\def\bq{\mathbf{Q}}

\def\bx{\mathbf{x}}
\def\by{\mathbf{y}}

\def\gl{\lambda}

\def\bz{\mathbb{Z}}
\def\fgl{\mathfrak{gl}}
\def\fg{\mathfrak{g}}
\def\fS{\mathfrak{S}}
\def\ft{\mathfrak{t}}
\def\mpn{\mathscr{P}_{m,n}}
\def\la{\langle}
\def\ra{\rangle}

\newenvironment{point}[2]%
  {\vspace{0.5\jot}\ifx*#2\let\pointlabel\relax\else\def\pointlabel{#2}\fi
   \refstepcounter{equation}\trivlist
   \item[\hskip\labelsep\theequation.
         \ifx\pointlabel\relax\else\space\pointlabel\space\fi]
   \ignorespaces #1
  \vspace{0.5\jot}}{\relax}
\numberwithin{equation}{section}
\swapnumbers
\newtheorem{theorem}[equation]{Theorem}
\newtheorem{lemma}[equation]{Lemma}
\newtheorem{proposition}[equation]{Proposition}

\theoremstyle{definition}
\newtheorem{definition}[equation]{Definition}
\theoremstyle{remark}
\newtheorem{remark}[equation]{Remark}
\begin{document}
\setlength{\itemsep}{1\jot}
\fontsize{13}{\baselineskip}\selectfont
\setlength{\parskip}{0.3\baselineskip}
\vspace*{0mm}
\title[Super Schur-Weyl duality]{\fontsize{9}{\baselineskip}\selectfont A super Frobenius formula for \\the characters of cyclotomic Hecke algebras}
\author{Deke Zhao}
\date{}
\address{School of Applied Mathematics, Beijing Normal University at Zhuhai, Zhuhai, 519087, China}
\email{ deke@amss.ac.cn}
%\thanks{Supported by the National Natural Science Foundation of China (Grant No. 11871107, 11571341, 11671234)}
\vspace*{-3mm}
\begin{abstract}We prove a super Frobenius formula for the characters of the cyclotomic Hecke algebras by applying the super Schur-Weyl reciprocity between the quantum superalgebras and cyclotomic Hecke algebras, which is a super analogue of the Frobenius formula in  (T. Shoji, J. Algebra \textbf{226}: 2000, 818--856) and a cyclotomic analogue of the super Frobenius formula in (H. Mitsuhashi, Linear Multilinear Algebra \textbf{58}: 2010, 941--955).
\end{abstract}
\maketitle
\vspace*{-5mm}
%%%%%%%%%%%%%%%%%%%%%%%%%%%%%%%%%%%%%%%%%%%%%%%%%%%%%%%%%%%%%%%%%%%%
%%%%%%%%%%%%%%%%%%%%%%%%%%%%%%%%%%%%%%%%%%%%%%%%%%%%%%%%%%%%%%%%%%%%
%\tableofcontents
%%%%%%%%%%%%%%%%%%%%%%%%%%%%%%%%%%%%%%%%%%%%%%%%%%%%%%%%%%%%%%%%%%%%
%%%%%%%%%%%%%%%%%%%%%%%%%%%%%%%%%%%%%%%%%%%%%%%%%%%%%%%%%%%%%%%%%%%%
\section{Introduction}
In \cite{Frobenius}, Frobenius gave a formula of computing the characters of the symmetric group, which is often referred as the Frobenius formula. In his study of representations of the general linear group, Schur \cite{Schur-d,Schur} showed the Frobenius formula can be obtained by the classical Schur-Weyl reciprocity. After Schur's classical work, Schur--Weyl reciprocity has been extended to various groups and algebras. Among them,  two  remarkable extensions for us are the super type extensions \cite{B-Regev,Serg} and the (cyclotomic) quantum type extension \cite{Jimbo,ATY,Hu,Moon,Mit}. In \cite{Zhao18}, we establish a super Schur-Weyl reciprocity between the cyclotomic Hecke algebra $\mathcal{H}$ and the quantum superalgebra $U_q(\fg)$.

Based on the Schur-Weyl reciprocity between the Iwahori-Hecke algebras of type $A$ and the quantum enveloping algebra of $\fgl(n)$ given by Jimbo \cite{Jimbo}, Ram \cite{Ram} gave a $q$-analogue of Frobenius formula for the characters of the Iwahori-Hecke algebras of type $A$. A super Frobenius formula for the characters of the Iwahori-Hecke algebras of type $A$ was given by Mitsuhashi in \cite{M2010} by virtue of the Schur-Weyl reciprocity between the Iwahori-Hecke algebras of type $A$ and the quantum superalgebra \cite{Moon,Mit}. An extension of Frobenius formula for the characters of cyclotomic Hecke algebra of type $G(m,1,n)$ is found in \cite{S} by applying the Schur-Weyl reciprocity between cyclotomic Hecke algebras and quantum algebras given in \cite{SS}.

 Motivated by these works, we prove a super Frobenius formulas for the characters of the cyclotomic Hecke algebras by applying the super Schur-Weyl reciprocity mentioned above. More precisely, let $\mpn$ be the set of multipartitions of $n$ and  let $g(\bmu)\in \mathcal{H}$ be the standard element of type $\bmu=(\mu^{(1)};\mu^{(2)};\ldots;\mu^{(m)})\in\mpn$ (see \ref{Equ:T-bmu}). It is known by Ariki and Koike \cite{AK} that the irreducible representations of $\mathcal{H}$ are indexed by $\mpn$. We denote by $S^{\bgl}$ the irreducible $\mathcal{H}$-module corresponding to $\bgl\in\mpn$ and by $\chi^{\bgl}$ its irreducible character.  Then the characters $\chi^{\bgl}$ of $\mathcal{H}$ are completely determined by their values on the standard element $g(\bmu)$  of type $\bmu$ for all $\bmu\in\mpn$ (see \cite[Proposition~7.5]{S}). For simplicity, we write $\chi^{\bgl}(\bmu)=\chi^{\bgl}(g(\bmu)))$. The super Frobenius formula (Theorem~\ref{Them:Frobenius}) states that \begin{equation*}
   q_{\bmu}(\bx/\by;q,\bq)=\sum_{\bgl\in\mpn}\chi^{\bgl}(\bmu)S_{\bgl}(\bx/\by),
 \end{equation*}
where $S_{\bgl}(\bx/\by)$ is the supersymmetric Schur function
associated to a multipartition $\bgl$ and $q_{\bmu}(\bx/\by;q,\bq)$ is a certain polynomial, which can be described by using super Hall-Littlewood functions $q_n(\bx/\by;t)$ (see Eq.~(\ref{Equ:Supersymm-multi}) and (\ref{Equ:q-bmu}) for the details). Let us remark that the super Frobenius formula enables us to derive a super Frobenius formula for the complex reflection group of type $G(m,1,n)$ by specializing $q=1$ and $Q_i=\varsigma^i$, where $\varsigma$  is a fixed primitive $m$-th root of unity (see Remark~\ref{Remark:Super-Frobenius}).

 This paper is organized as follows. In Section~\ref{Sec:Quantum-superalgebras} we review briefly the definitions of quantum superalgebra and cyclotomic Hecke algebras, and  the super Schur-Weyl reciprocity.  Section~\ref{Sec:super-symm} devotes to introduce the supersymmetric Schur functions and power sum supersymmetric functions associated to multipartitions and the super Hall-Littlewood functions. In particular, we prove a super-analogue of \cite[Appendix B~(9.5)]{Macdonald} and \cite[Proposition~6.6]{S} (see Proposition~\ref{Prop:super-Schur-power}). The super Frobenius formula for the characters of cyclotomic Hecke algebras is proved in last section.

 Throughout the paper, we assume that  $\mathbb{K}=\mathbb{C}(q,\bq)$ is the field of rational function in indeterminates $q$ and  $\bq=(Q_1, \ldots, Q_m)$.
 For fixed non-negative $k,\ell$ with $k+\ell>0$, we define the parity function  $i\mapsto \overline{i}$ by \begin{equation*} \overline{i}=\left\{\begin{array}{ll}\overline{0}, &\hbox{ if }1\le i\leq k;\\\overline{1}, & \hbox{ if }k<i\leq k+\ell.\end{array}\right.\end{equation*}
Assume that $k_1, \ldots, k_m$, $\ell_1, \ldots, \ell_m$ are non-negative integers satisfying  $\sum_{i=1}^mk_i=k$, $\sum_{i=1}^m\ell_i=\ell$ and denote by $\bk=(k_1, \ldots, k_m)$, $\bl=(\ell_1, \ldots, \ell_m)$. For $i=1, \ldots, m$, we define $d_i=\sum_{j\leq i}k_j+\ell_j$.

\subsection*{Acknowledgements}Part of this work was carried out while the author was visiting the Chern Institute of Mathematics at the Nankai University. The authors would like to thank Professor Chengming Bai for his hospitality during his visit. The author is very grateful to the anonymous referee for his/her valuable comments and suggestions which greatly improved this paper.  The author was supported partly by the National Natural Science Foundation of China (Grant No. 11571341, 11671234, 11871107).
%%%%%%%%%%%%%%%%%%%%%%%%%%%%%%%%%%%%%%%%%%%%%%%%%%%%%%%%%%%%%%%%%%%%
%%%%%%%%%%%%%%%%%%%%%%%%%%%%%%%%%%%%%%%%%%%%%%%%%%%%%%%%%%%%%%%%%%%%
\section{Preliminaries}\label{Sec:Quantum-superalgebras}
In this section we recall the definitions of quantum superalgebra and cyclotomic Hecke algebras. The notations and some known facts which are need in the following sections are also established.

\begin{point}{}*
Recall that  the Lie superalgebra $\mathfrak{gl}(k|\ell)$ is the $(k+\ell)\times (k+\ell)$ matrices with $\mathbb{Z}_2$-gradings given by \begin{eqnarray*} \mathfrak{gl}(k|\ell)_{\bar 0}&=&\left\{\left.\left(\begin{array}{cc}\mathbf{A} & \mathbf{0} \\ \mathbf{0} & \mathbf{D}\end{array}\right)\right|\mathbf{A}=(a_{ij})_{1\leq i,j\leq k}, \mathbf{D}=(d_{ij})_{k< i,j\leq k+\ell}\right\},\\ \mathfrak{gl}(k|\ell)_{\bar 1}&=&\left\{\left.\left(\begin{array}{cc}\mathbf{0}& \mathbf{B} \\ \mathbf{C} & \mathbf{0}\end{array} \right)\right|\mathbf{B}=(b_{ij})_{1\leq i\leq k}^{k<j\leq k+\ell}, \mathbf{C}=(c_{ij})_{k<i\leq k+\ell}^{1\leq j\leq k}\right\}                                      \end{eqnarray*}
and Lie bracket product defined by $$[\mathbf{X},\mathbf{Y}]:=\mathbf{XY}-(-1)^{\overline{\mathbf{X}}\,\overline{\mathbf{Y}}}\mathbf{YX}$$ for homogeneous $\mathbf{X},\mathbf{Y}$.

For $a,b=1,\ldots, k+\ell$, we denote by  $\mathbf{E}_{a,b}$ the elementary $(k+\ell)\times (k+\ell)$ matrix with 1 in the $(a,b)$-entry and  zero in all other entries.  Let $\epsilon_i: \mathfrak{gl}(k|\ell)\rightarrow \mathbb{C}$ be the linear function on $\mathfrak{gl}(k|\ell)$ defined by \begin{equation*}
  \epsilon_i(\mathbf{E}_{a,b})=\delta_{i,a}\delta_{a,b} \text{ for }i, a,b\in [1, k+\ell].
\end{equation*}
The free abelian group $P=\bigoplus\limits_{i=1}^{k+\ell}\mathbb{Z}\epsilon_i$ (resp. $P^{\vee}=\bigoplus\limits_{i=1}^{k+\ell}\mathbb{Z}\mathbf{E}_{b,b}$) is called the \textit{weight lattice} (resp. \textit{dual weight lattice}) of $\mathfrak{gl}(k|\ell)$, and there is a symmetric bilinear form
$(\,,\,)$ on $\mathfrak{h}^*=\mathbb{C}\otimes_{\mathbb{Z}}P$ defined by
\begin{equation*}
  (\epsilon_i,\epsilon_j)=(-1)^{\overline{i}}\delta_{i,j} \text{ for }i,j\in [1, k+\ell].
\end{equation*}
Then the simple roots of $\mathfrak{gl}(k,\ell)$ are $\alpha_i=\epsilon_i-\epsilon_{i+1}$, $i=1, \ldots, k+\ell-1$.
We have positive root system $\Phi^{+}=\{\alpha_{i,j}=\epsilon_i-\epsilon_j|1\leq i<j\leq k+\ell\}$ and negative root system $\Phi^{-}=-\Phi^{+}$. Define $\overline{\alpha}_{i,j}=\overline{i}+\overline{j}$ and call $\alpha_{i,j}$ is an even (resp. odd) root if $\overline{\alpha}_{i,j}=\overline{0}$ (resp. $\overline{1}$). Note that $\alpha_{k}$ is the only odd simple root. Denote by $\langle\cdot,\cdot\rangle$ the natural pairing between $P$ and $P^{\vee}$. Then the simple coroot $\alpha^{\vee}_i$ corresponding to $\alpha_i$ is the unique element in $P^{\vee}$ satisfying \begin{equation*}\langle \alpha^{\vee}_i,\lambda\rangle=(-1)^{\overline{i}}(\alpha_i, \lambda)\text{ for all }\lambda\in P. \end{equation*}
\end{point}

\begin{definition}[\cite{Zhang}]The \textit{quantum superalgebra} $U_q(\mathfrak{gl}(k|\ell)$, i.e.,  \textit{quantized universal enveloping algebra} of $\fgl(k|\ell)$  is the unitary superalgebra over $\mathbb{K}$ generated by the homogeneous  elements
\begin{equation*}
  E_1, \ldots, E_{k+\ell-1}, F_1, \ldots, F_{k+\ell-1}, K_1^{\pm1}, \ldots, K_{k+\ell}^{\pm1}
\end{equation*}
with a $\mathbb{Z}_2$-gradation by letting $\overline{E}_k=\overline{F}_k=\overline{1}$,
 $\overline{E}_a=\overline{F}_a=\overline{0}$ for $a\neq k$, and $\overline{{K_i}^{\pm1}}=\overline{0}$. These generators satisfy the following relations:
\begin{enumerate}
  \item[(Q1)] $K_aK_b=K_bK_a, K_aK_a^{-1}=K_a^{-1}K_a=1$;
  \item[(Q2)] $K_aE_b=q^{\la\alpha^{\vee}_a,\alpha_b \ra}E_bK_a$;
\item[(Q3)] $E_aE_b=E_bE_a, F_aF_b=F_bF_a$  if $|a-b|>1$;
\item[(Q4)]$[E_a,F_b]=\delta_{a,b}\frac{\widetilde{K}_a-\widetilde{K}_a^{-1}}{q_a-q_a^{-1}}$, where $q_a=q^{(-1)^{\overline{a}}}$ and $\widetilde{K}_a=K_aK^{-1}_{a+1}$;
\item[(Q5)] For $a\neq k$ and $|a-b|>1$,
\begin{eqnarray*}
 && E_a^2E_b-(q_a+q_a^{-1})E_aE_bE_a+E_bE_a^2=0,\\
&&F_a^2F_b-(q_a+q_a^{-1})F_aF_bF_a+F_bF_a^2=0;
\end{eqnarray*}
\item[(Q6)] $E_k^2=F^2_k=0$,\\
$E_k\!\left(\!E_{k\!-\!1}E_kE_{k\!+\!1}\!\!+\!\!E_{k\!+\!1}E_{k}E_{k\!-\!1}\!\right)\!\!-
\!\!\left(\!q\!\!+\!\!q^{-\!1}\!\right)\!E_kE_{k\!-\!1}E_{k\!+\!1}E_k
\!\!+\!\!\left(\!E_{k\!-\!1}E_kE_{k\!+\!1}\!\!+\!\!E_{k\!+\!1}E_kE_{k\!-\!1}\!\right)\!E_k$,\\
$F_k\!\left(\!F_{k\!-\!1}F_kF_{k\!+\!1}\!\!+\!\!F_{k\!+\!1}F_{k}F_{k\!-\!1}\!\right)\!\!-
\!\!\left(\!q\!\!+\!\!q^{-\!1}\!\right)\!F_kF_{k\!-\!1}F_{k\!+\!1}F_k
\!\!+\!\!\left(\!F_{k\!-\!1}F_kF_{k\!+\!1}\!\!+\!\!F_{k\!+\!1}F_kF_{k\!-\!1}\!\right)\!F_k$.
\end{enumerate}
  \end{definition}
It is known that $U_q(\mathfrak{gl}(k|\ell))$  is a Hopf superalgebra with comultiplication $\Delta$ defined by
\begin{align*}
  &\Delta(K_i^{\pm1})=K_i^{\pm 1}\otimes K_i^{\pm 1},\\
 &\Delta(E_i)=E_i\otimes \widetilde{K}_i+1\otimes E_i, \\
& \Delta(F_i)=F_i\otimes 1+ \widetilde{K}_i^{-1}\otimes F_i.
  \end{align*}
\begin{point}{}*
Let $V$ be a superspace over $\mathbb{K}$ with $\dim V= k|\ell$, that is, $V=\mathbb{C}^{k|\ell}\otimes_{\mathbb{C}}\mathbb{K}$, and let $\mathfrak{B}=\{v_1,\ldots, v_{k+\ell}\}$ be its homogeneous basis with $\bar{v}_i=\bar{i}$.  Then the \textit{vector representation} $\Psi$ of $U_q(\mathfrak{gl}(k|\ell))$ on $V$ is defined by
\begin{eqnarray*}
  &&\Psi(E_i)v_{j}=\left\{
                     \begin{array}{ll}
                       (-1)^{\overline{v}_j}v_{j-1}, &\quad \hbox{ if }j=i+1; \\
                       0, & \quad\hbox{ others.}
                     \end{array}
                   \right.;\\
&&\Psi(F_i)v_{j}=\left\{
                     \begin{array}{lll}
                       (-1)^{\overline{v}_j}v_{j+1}, &\quad \hbox{ if }j=i; \\
                       0, & \quad\hbox{ others.}
                     \end{array}
                   \right.\\
&&\Psi(K_i^{\pm 1})(v_j)=\left\{
                     \begin{array}{ll}
                       (-1)^{\overline{v}_j}q^{\pm1}v_{j}, & \quad\hbox{ if }j=i; \\
                       0, & \quad\hbox{ others.}
                     \end{array}
                   \right.
\end{eqnarray*}

For a positive integer $n$, we can define inductively a superalgebra homomorphism
\begin{equation*}
  \Delta^{(n)}: U_q(\mathfrak{gl}(k|\ell))\rightarrow U_q(\mathfrak{gl}(k|\ell))^{\otimes n},\quad \Delta^{(n)}=(\Delta^{(n-1)}\otimes \mathrm{id})\circ \Delta
\end{equation*}
 for each $n\geq 3$, where $\Delta^{(2)}=\Delta$. Therefore, $\Psi$ can be extended to the representation on tensor space $V^{\otimes n}$ via the Hopf superalgebra structure of $U_q(\mathfrak{gl}(k|\ell))$ for each $n$, we denote it by $\Psi^{\otimes n}$.
According to \cite[Proposition~3.1]{BKK}, $(\Psi, V)$ is an irreducible highest weight module $V(\epsilon_1)$ with highest weight $\epsilon_1$ and $V^{\otimes n}$ is complete reducible for all $n$.
\end{point}

\begin{point}{}* Now suppose that $V=V^{(1)}\oplus\cdots\oplus V^{(m)}$, where $V^{(i)}$ is a subsuperspace of $V$ with $\dim V^{(i)}=k_i|\ell_i$, and assume that
\begin{equation*}
  \mathfrak{B}^{(i)}=\left\{v^{(i)}_1, \ldots, v_{k_i+\ell_i}^{(i)}\right\}
\end{equation*}
is a basis of $V^{(i)}$ such that $v_{a}^{(i)}$ (resp. $v_{b}^{(i)}$) being of degree $\bar{0}$ (resp. $\bar{1}$) for $1\leq a\leq k_i$ (resp. $k_i<b\leq k_i+\ell_i$) for each $i$ such that $1\leq i\leq m$.
We say that the vectors in $\mathfrak{B}^{(i)}$ are of color $i$ and we linearly order the vectors $v_1^{(1)}, \ldots, v^{(m)}_{k_m+\ell_m}$ by the rule\begin{eqnarray*}
  v_{a}^{(i)}<v_{b}^{(j)}&\text{ if and only if }& i<j\text{ or }i=j\text{ and }a<b.                                                          \end{eqnarray*}We may and will identify $v_1^{(1)}$, $\ldots$, $v^{(m)}_{k_m+\ell_m}$ with $v_1$, $\cdots$, $v_{k+\ell}$ as follows:
\begin{equation*}\label{Equ:Basis-index}
  \begin{array}{ccccccccccccc}
  v_{1}^{(1)}& \cdots & v_{k_1}^{(1)}& v_{k_1\negmedspace+\negmedspace1}^{(1)}&\cdots& v_{k_1\negmedspace+\negmedspace\ell_1}^{(1)}&\cdots&
  v_{1}^{(m)}& \cdots & v_{k_m}^{(m)}& v_{k_m+1}^{(m)}&\cdots& v_{k_m+\ell_m}^{(m)} \\
  \updownarrow& \vdots & \updownarrow & \updownarrow & \vdots & \updownarrow&\vdots&\updownarrow&\vdots&\updownarrow&\updownarrow&\vdots&\updownarrow \\
  v_1& \cdots & v_{k_1} & v_{k+1} & \cdots& v_{k\negmedspace+\negmedspace\ell_1}&\cdots&v_{k\negmedspace-\negmedspace k_m\negmedspace+\negmedspace1}&\cdots&v_k&v_{d_m\negmedspace-\negmedspace\ell_m\negmedspace+
  \negmedspace1}&\cdots&v_{k+\ell},
\end{array}
\end{equation*}
that is, $\mathfrak{B}=\mathfrak{B}^{(1)}\sqcup\cdots\sqcup\mathfrak{B}^{(m)}$.

Let $\mathscr{I}(n;k|\ell)=\{\bi=(i_1, \ldots, i_n)|1\leq i_t\leq k+\ell, 1\leq t\leq n\}$.
For $\bi=(i_1, \ldots, i_n)\in \mathscr{I}(n;k|\ell)$, we write $v_{\bi}=v_{i_1}\otimes \cdots\otimes v_{i_n}$ and put $c_a(v_{\bi})=b$ if $v_{i_a}$ is of color $b$. Then $\mathfrak{B}^{\otimes n}=\{v_{\bi}|\bi\in \mathscr{I}(n;k|\ell)\}$ is a basis of $V^{\otimes n}$. From now on, we will identify $\mathfrak{B}^{\otimes n}$ with $\mathscr{I}(n;k|\ell)$, that is, we will write $v_{\bi}$ by $\bi$, $\overline{v}_i$ by $\overline{i}$, $c_a(v_{\bi})$ by $c_a(\bi)$, etc., for $\bi\in \mathscr{I}(n;k|\ell)$, and use these notations freely.

Clearly, the Lie superalgebra $\fgl(k_i|\ell_i)$ can be viewed as a subalgebra of $\fgl(k|\ell)$ for each $i$ such that $1\leq i\leq m$. Therefore the Lie superalgebra $\mathfrak{g}=\mathfrak{gl}(k_1|\ell_1)\oplus\cdots\oplus\mathfrak{gl}(k_m|\ell_m)$ can be viewed as a subalgebra of $\fgl(k|\ell)$ and its quantum superalgebra $U_q(\mathfrak{g})$ can be naturally embedded in  $U_q(\mathfrak{gl}(k,\ell))$ as a $\mathbb{K}$-subalgebra generated by
\begin{equation*}\label{Equ:Uq(g)-generators}
 \mathscr{G}=\left\{ E_a, F_a, K_{b}^{\pm1}\mid a\in \{1,2, \ldots, d_m\}\backslash\{d_1, d_2, \ldots,d_m\}, 1\leq b\leq d_m\right\}.
\end{equation*}
Hence $U_q(\mathfrak{g})$ acts on $V^{\otimes n}$ by the restriction of $\Psi^{\otimes n}$, we also denote the action by $\Psi^{\otimes n}$.
\end{point}

\begin{point}{}* Recall that a composition (resp. partition) $\gl=(\gl_1, \gl_2, \ldots)$ of $n$, denote $\gl\models n$ (resp. $\gl\vdash n$) is a sequence (resp. weakly decreasing sequence) of  nonnegative integers such that $|\gl|=\sum_{i\geq1}\gl_i=n$ and we write $\ell(\gl)$ the {\it length} of $\gl$, i.e. the number of nonzero parts of $\gl$. A {\it multicomposition} (resp. {\it multipartition}) of $n$ is an ordered tuple $\bgl=(\gl^{(1)}; \ldots; \gl^{(m)})$ of compositions (resp. partitions) $\lambda^{i}$ such that $n=\sum_{i=1}^m|\lambda^{i}|$. We denote by $\mpn$ the set of all multipartitions of $n$ and by $\mathscr{C}(n;m)$ the set of all multicompositions of $n$ with length not greater than $m$.

A partition $\gl=(\gl_1, \gl_2, \cdots)\vdash n$ is said to be a \textit{$(k, \ell)$-hook partition} of $n$ if $\gl_{k+1}\leq \ell$. We let  $H(k,\ell;n)$ denote the set of all $(k,\ell)$-hook partitions of $n$, that is
\begin{eqnarray*}\label{Def:hook}
H(k,\ell;n)=\{\gl=(\gl_1,\gl_2,\cdots)\vdash n\mid \gl_{k+1}\le \ell\}.
\end{eqnarray*}
A multipartition $\bgl=(\gl^{(1)}; \ldots; \gl^{(m)})$ of $n$ is said to be a \textit{$(\bk,\bl)$-hook multipartition} of $n$ if $\gl^{(i)}$ is a $(k_i,\ell_i)$-hook partition for all $i=1, \ldots,m$. We denote by $H(\bk|\bl; m,n)$ the set of all  $(\bk,\bl)$-hook multipartition of $n$.

It is known that  the irreducible representations of $U_q(\mathfrak{gl}(k,\ell))$ occurring in $V^{\otimes n}$ are parameterized by the $(k,\ell)$-hook partitions of $n$. Since $U_q(\mathfrak{g})=U_q(\mathfrak{gl}(k_1,\ell_1))\otimes\cdots\otimes U_q(\mathfrak{gl}(k_m,\ell_m))$, irreducible representations of $U_q(\mathfrak{g})$ occurring in $V^{\otimes n}$ are parameterized by the $(\bk,\bl)$-hook multipartitions of $n$, that is,   the irreducible representations of $U_q(\mathfrak{g})$ occurring in $V^{\otimes n}$ are parameterized by $H(\bk|\bl;m,n)$.
\end{point}

\begin{point}{}*\label{subsec:G(m,1,n)}
 Let $W_{m,n}$ be the complex reflection group of type $G(m,1,n)$. According to \cite{shephard-toda}, $W_{m,n}$ is generated by $s_1, s_2, \dots, s_{n}$ with relations
  \begin{align*}&s_1^m=1,\quad s_2^2=\cdots=s_{n}^2=1,&&\\
  &s_1 s_2s_1 s_2=s_2s_1 s_2s_1,&&\\
 & s_is_j=s_js_i,&&\text{ if } |i-j|>1,\\
&s_is_{i+1}s_i=s_{i+1}s_{i}s_{i+1},&& \text{ for }2\leq i< n.\end{align*}
 It is well-known that $W_{m,n}\cong(\mathbb{Z}/m\bz)^{n}\rtimes \mathfrak{S}_{n}$, where $s_2, \dots, s_{n}$ are generators of the symmetric group $\mathfrak{S}_{n}$ of degree $n$ corresponding to transpositions $(1\,2)$, $\ldots$, $(n\!-\!1\,n)$. For $a=1,\ldots, n$,  let $t_a=s_{a}\cdots s_2s_1s_2\cdots s_{a}$. Then $t_1, \cdots, t_n$ are generators of $(\mathbb{Z}/m\bz)^{n}$ and any element $w\in W_{m,n}$ can be written in a unique way as $w=t_1^{c_1}\cdots t_n^{c_n}\sigma$, where $\sigma\in \mathfrak{S}_{n}$ and $c_i$ are integers such that $0\leq c_i<m$.
 \end{point}

\begin{point}{}*
The \textit{Ariki-Koike algebra} \cite{AK}, i.e., the \textit{cyclotomic Hecke algebra} $\mathcal{H}=H_{m,n}(q,\mathbf{Q})$ associated to $W_{m,n}$ is the unital
associative $\mathbb{K}$-algebra  generated by
$g_1,g_2,\dots,g_{n}$ and subject to relations
\begin{align*}&(g_1-Q_1)\dots(g_1-Q_m)=0,&&\\
&g_1g_2g_1g_2=g_2g_1g_2g_1,&&\\
&g_i^2=(q-q^{-1})g_i+1, &&\text{ for }2\leq i\leq n,\\
&g_ig_j=g_jg_i, &&\text{ for }|i-j|>2,\\
 &g_ig_{i+1}g_i=g_{i+1}g_{i}g_{i+1}, &&\text{ for }2\leq i<n.\end{align*}
If $s_{i_1}s_{i_2}\cdots s_{i_k}$ is a reduced expression for $\sigma\in \mathfrak{S}_n$. Then $g_{\sigma}:=g_{i_1}g_{i_2}\cdots g_{i_k}$ is independent of the choice of reduced expression and  $\{g_{\sigma}|\sigma\in \mathfrak{S}_n\}$ is a linear basis of the subalgebra $H_n(q)$ of $\mathcal{H}$ generated by $g_2, \ldots, g_{n}$.
\end{point}

\begin{point}{}*
For $s_a=(a-1,a)\in \fS_n$ and
$\bi=(i_1, \ldots, i_{a-1},i_{a},\ldots, i_n)$, we define the following right action
\begin{equation*}\label{Equ:is_a}
  \bi s_a:=(i_1, \ldots, i_{a-2}, i_{a}, i_{a-1}, i_{a+1}, \ldots, i_n).
\end{equation*}
Following Sergeev \cite{Serg} and Berele-Regev \cite{B-Regev}, we define a right action $\phi:\mathbb{C}\mathfrak{S}_n\rightarrow  \mathrm{End}_{\mathbb{K}}(V^{\otimes n})$ by
\begin{eqnarray*}\label{Equ:W_mn-sign}
 \phi(s_a)(\bi)&:=&\left\{\begin{array}{ll}\vspace{1\jot}
 (-1)^{\overline{i}_{a-1}}\bi,& \text{if }i_{a-1}=i_{a};\\
(-1)^{\overline{i}_{a-1}\overline{i}_{a}}\bi s_a,& \text{if }i_{a-1}\neq i_{a}.
                         \end{array}\right.
\end{eqnarray*}
For $a=2, \ldots, n$, we define the endomorphisms $T_a, S_a\in \mathrm{End}_\mathbb{K}(V^{\otimes n})$ as follows: \begin{eqnarray}\label{Equ:Ta-action}
 && T_a(\bi):=\left\{\begin{array}{ll}\vspace{2\jot}(q-q^{-1})\bi
 +\phi(s_a)(\bi),& \text{if }i_{a-1}<i_{a};\\ \vspace{2\jot}
 \frac{q-q^{-1}}{2}\bi+\frac{q+q^{-1}}{2}\phi(s_a)(\bi),& \text{if }i_{a-1}=i_{a}; \\
  \phi(s_a)(\bi),&\text{if } i_{a-1}>i_{a}.
                         \end{array}\right.\end{eqnarray}
\begin{eqnarray*}&&S_a(\bi):=\left\{\begin{array}{ll}
 T_a(\bi), & \hbox{ if } c_{a-1}(\bi)= c_{a}(\bi);\\
 \phi(s_a)(\bi), & \hbox{ if } c_{a-1}(\bi)\neq c_{a}(\bi);\end{array}\right.
\end{eqnarray*}
and let $\Omega_j(\bi):=Q_{c_j(\bi)}\bi$ for $j=1, \ldots, n$. Finally, we define $T_1\in \mathrm{End}_\mathbb{K}(V^{\otimes n})$ by letting
\begin{eqnarray*}\label{Equ:T_0-action}
T_1(\bi):&=&T_{2}^{-1}\cdots T_{n}^{-1}S_n\cdots S_2\Omega_1(\bi).
\end{eqnarray*}
\end{point}

It is shown in \cite[Theorem~4.12]{Zhao18} that $\Phi: g_i\mapsto T_i$ ($1\leq i\leq n$) is an (super) representation of $\mathcal{H}$ on $V^{\otimes n}$. Furthermore, the following Schur-Weyl reciprocity  between $U_q(\fg)$ and $\mathcal{H}$ holds.

\begin{theorem}[\protect{\cite[Theorem~5.13]{Zhao18}}]\label{Them:Schur-Weyl}The $\Psi^{\otimes n}(U_q(\mathfrak{g}))$ and $\Phi(\mathcal{H})$ are mutually the fully centralizer algebras of each other, i.e.,
  \begin{eqnarray*}
      \Psi^{\otimes n}(U_q(\mathfrak{g}))=\mathrm{End}_{\mathcal{H}}(V^{\otimes n}), &\qquad& \Phi(\mathcal{H})=\mathrm{End}_{U_q(\mathfrak{g})}(V^{\otimes n}).
  \end{eqnarray*}
More precisely,   there is a $U_q(\fg)\text{-}\mathcal{H}$-bimodule isomorphism
\begin{equation*}\label{Equ:V^n-decom}
  V^{\otimes n}\cong \bigoplus_{\bgl\in H(\bk|\bl;m,n)}V_{\bgl} \otimes S^{\bgl},
\end{equation*}
 where $V_{\bgl}$ is the irreducible $U_q(\fg)$-module indexed by $\bgl$.
\end{theorem}

\begin{point}{}*
Let $\Delta$ be the determinant of the Vandermonde matrix $V(\mathbf{Q})$ of degree $m$ with $(a,b)$-entry $Q_b^a$ for $1\leq b\leq m$, $0\leq a<m$. Clearly, we can write $V(\mathbf{Q})^{-1}=\Delta^{-1}V^*(\mathbf{Q})$, where $V^*(\mathbf{Q})=(v_{ba}(\mathbf{Q}))$ and $v_{ba}(\mathbf{Q})$ is a polynomial in $\mathbb{Z}[\mathbf{Q}]$.

For $1\leq c\leq m$, we define a polynomial $F_c(X)$ with a variable $X$ with coefficients in $\mathbb{Z}[\mathbf{Q}]$ by
\begin{equation*}
  F_c(X)=\sum_{0\leq i<m}v_{ci}(\mathbf{Q})X^i.
\end{equation*}
%The following formula is easily obtained from the definition $F_c(Q_{c'})=\delta_{cc'}\Delta$.

Following \cite[\S 3.6]{S}, let $\mathcal{H}^{\natural}$ be the associative algebra over $\mathbb{K}$ generated by $g_2, \ldots, g_{n}$ and $\xi_1, \ldots, \xi_n$ subject to the following relations:
\begin{align*}
                       & (g_i-q)(g_i+q^{-1})=0,  && 2\leq i\leq n,\\
          &(\xi_i-Q_1)\cdots(\xi_i-Q_m)=0, && 1\leq i\leq n, \\
           &g_ig_{i+1}g_{i}=g_{i+1}g_ig_{i+1},&&  2\leq i< n,\\
            &g_ig_j=g_jg_i,&& |i-j|\geq 2, \\
         & \xi_i\xi_j=\xi_j\xi_i,&& 1\leq i,j\leq n, \\
         &g_j\xi_i=\xi_ig_j,&&  i\neq j-1, j,\\
          &g_j\xi_{j}=\xi_{j-1}g_j+\Delta^{-2}\sum_{a<b}(Q_{a}-Q_{b})(q-q^{-1})F_{a}(\xi_{j-1})F_{b}(\xi_j),& \\
 &g_j\xi_{j-1}=\xi_{j}g_j-\Delta^{-2}\sum_{a<b}(Q_{a}-Q_{b})(q-q^{-1})F_{a}(\xi_{j-1})F_{b}(\xi_j).&
            \end{align*}
Then $\mathcal{H}^{\natural}$ is isomorphic to $\mathcal{H}$ owing to \cite[Theorem~3.7]{S}. Moreover, the linear map  $\widetilde{\Phi}: \mathcal{H}^{\natural}\rightarrow \mathrm{End}(V^{\otimes n})$ defined by $g_i\mapsto T_i$, $\xi_j\mapsto \Omega_j$ for $2\leq i\leq n$, $1\leq j\leq n$,  is a (super) representation of $\mathcal{H}^{\natural}$, which is isomorphic to the (super) representation $(\Phi, V^{\otimes n})$ of $\mathcal{H}$.
\end{point}

\begin{remark} Let $\phi: \mathbb{K}\rightarrow \mathbb{C}$ be the specialization homomorphism defined by $\phi(q)=1$ and $\phi(Q_i)=\varsigma^i$ for each $i$, where $\varsigma$ is a fixed primitive $m$-th root of unity. By the specialization homomorphism $\phi$, one obtains $\mathbb{C}\otimes_{\mathbb{K}}\mathcal{H}\backsimeq \mathbb{C}W_{m,n}$ and $\mathbb{C}\otimes_{\mathbb{K}}V^{\otimes n}\simeq \overline{V}^{\otimes n}$ where $\overline{V}=\oplus_{i=1}^m\mathbb{C}^{k_i|\ell_i}$. Moreover,  we yield the Schur-Weyl duality between $U(\fg)$ and $\mathbb{C}W_{m,n}$ by applying the specialization.
\end{remark}
\section{Supersymmetric functions}\label{Sec:super-symm}

In this section we introduce the supersymmetric Schur functions and power sum supersymmetric functions indexed by multipartitions. We will follow \cite{Macdonald} with respect to our notation about symmetric functions unless otherwise stated.

\begin{point}{}*
 Let $\mathbf{x}, \mathbf{y}$ be sets of $k, \ell$ indeterminates respectively as follows
\begin{eqnarray*}
  \mathbf{x}^{(i)}&=&\left\{x^{(i)}_1, \ldots, x_{k_i}^{(i)}\right\},\quad 1\leq i\leq m,\\
  \mathbf{y}^{(i)}&=&\left\{y^{(i)}_1, \ldots, y_{\ell_i}^{(i)}\right\},\quad 1\leq i\leq m,\\
    \mathbf{x}&=&\mathbf{x}^{(1)}\cup\cdots\cup \mathbf{x}^{(m)},\\
    \mathbf{y}&=&\mathbf{y}^{(1)}\cup\cdots\cup \mathbf{y}^{(m)}.
\end{eqnarray*}
We linearly order the indeterminates $x_1^{(1)}, \ldots, x^{(m)}_{k_m}$, $y_1^{(1)}, \ldots, y^{(m)}_{\ell_m}$  by the rule\begin{eqnarray*}  x_{a}^{(i)}<x_{b}^{(j)}&\text{ if and only if }& i<j\text{ or }i=j\text{ and }a<b,\\ y_{a}^{(i)}<y_{b}^{(j)}&\text{ if and only if }& i<j\text{ or }i=j\text{ and }a<b.                                                         \end{eqnarray*}
We identify the indeterminates  $x_1^{(1)}$, $\ldots$, $y^{(m)}_{\ell_m}$ with the indeterminates $x_1$, $\ldots$, $x_{k}, y_1, \ldots, y_{\ell}$ as follows:
\begin{eqnarray*}\label{Equ:Basis-index}
  &\begin{array}{cccccc}
  x_{1}^{(1)}& \cdots & x_{k_1}^{(1)} &x_{1}^{(2)} & \cdots&x_{k_m}^{(m)} \\
  \updownarrow& \vdots & \updownarrow & \updownarrow  & \vdots&\updownarrow \\
  x_1& \cdots & x_{k_1} & x_{k_1+1} &  \cdots&x_{k},
\end{array}\quad&\quad\begin{array}{cccccc}
  y_{1}^{(1)}& \cdots & y_{\ell_1}^{(1)} &y_{1}^{(2)}&  \cdots&y_{\ell_m}^{(m)} \\
  \updownarrow& \vdots & \updownarrow & \updownarrow  & \vdots&\updownarrow \\
  y_1& \cdots & y_{\ell_1} & y_{\ell_1+1} &  \cdots&y_{\ell}.
\end{array}
\end{eqnarray*}
\end{point}

\begin{point}{}*
Let $\Lambda_{k,\mathbb{Z}}=\mathbb{Z}[x_1, \ldots, x_k]^{\mathfrak{S}_k}$ be the ring of symmetric functions of $k$ variables and $\Lambda_k=\Lambda_m\otimes_{\mathbb{Z}}\mathbb{Q}$. Then $\Lambda_k$ is graded with respect to the degree of these polynomials. That is $\Lambda_k=\bigoplus_{n}\Lambda_k^n$ where $\Lambda_k^n$ is the vector space of homogeneous symmetric polynomials of total degree $n$.

The \textit{power sum symmetric functions} \cite[$\mathrm{P}_{23}$]{Macdonald}
\begin{equation*}
  p_0(\bx)=1\text{ and } p_a(\bx)=\sum_{i=1}^{k}x_{i}^a\text{ for } a=1,2,\cdots
\end{equation*}
generate $\Lambda_{k}$ under the power series multiplication, whilst $\Lambda_k^n$ has a basis
\begin{equation*}
  p_{\lambda}(\bx):=p_{\lambda_1}(\bx)p_{\lambda_2}(\bx)\cdots \text{ with }\lambda\vdash n.
\end{equation*}

An alternative basis of $\Lambda_k^n$ is given by the \textit{Schur functions} $S_{\lambda}(\bx)$ with $\lambda\vdash n$, that is,
\begin{equation*}
  p_{\mu}(\bx)=\sum_{\lambda\vdash n}\chi^{\lambda}(\mu)S_{\lambda}(\bx) \quad\text { and }\quad
  S_{\lambda}(\bx)=\sum_{\mu\vdash n}Z_{\mu}^{-1}\chi^{\lambda}(\mu)p_{\mu}(\bx),
\end{equation*}
here $\chi^{\lambda}$ is the irreducible characters of $\mathbb{C}\mathfrak{S}_n$ corresponding to $\lambda\vdash n$, $\chi^{\lambda}(\mu)$ is the value of $\chi^{\mu}$ at elements of cycle type $\mu$ and $Z_{\mu}$ is the order of the centralizer in $\mathfrak{S}_n$ of an element of cycle type $\mu$.
Let us remark that $S_{\lambda}(\bx)=0$ is $\ell(\lambda)>k$ (see e.g. \cite[Equ.~(3.8)]{King}).

Now we denote by $\Lambda_{k,\ell}$ the ring of polynomials in $x_1, \ldots, x_k,y_1, \ldots, y_{\ell}$, which are separately symmetric in $\mathbf{x}$'s and $\mathbf{y}$'s, namely
\begin{equation*}
  \Lambda_{k,\ell}=\mathbb{Q}[x_1, \ldots, x_k]^{\mathfrak{S}_k}\otimes_{\mathbb{Q}}\mathbb{Q}[y_1, \ldots, y_\ell]^{\mathfrak{S}_\ell}.
\end{equation*}

Following  \cite[\S6]{B-Regev} or \cite[$\mathrm{P}_{280}$]{King}, the \textit{supersymmetric Schur function} $S_{{\lambda}^{(i)}}(\bx/\by)\in \Lambda_{k,\ell}$  corresponding to a partition $\lambda^{(i)}=\left(\lambda_1^{(i)}, \lambda_2^{(i)}, \ldots\right)$ is defined as
\begin{equation*}
  S_{\lambda^{(i)}}(\bx^{(i)}/\by^{(i)}):=\sum_{\mu\subset\lambda}(-1)^{|\lambda-\mu|}
  S_{\mu}(\bx^{(i)})S_{\lambda'/\mu'}(\by^{(i)}),
\end{equation*}
where $S_{\lambda'/\mu'}(\by^{(i)})$ is the skew Schur function associated to the conjugate $\lambda'/\mu'$  of the skew partition $\lambda/\mu$. Note that the skew Schur function  $S_{\eta/\theta}(\by^{(i)})$ can be calculated by $S_{\eta/\theta}(\by^{(i)})=\sum_{\nu}c^{\eta}_{\theta\nu}S_{\nu}(\by^{(i)})$, where the coefficients $c^{\eta}_{\theta\nu}$ are determined by the Littlewood-Richardson rule in the product of Schur functions (see \cite[$\mathrm{P}_{143}$]{Macdonald}). It should be noted that $S_{\lambda^{(i)}}(\bx/\by)=0$ unless $\lambda^{(i)}$ is a $(k,\ell)$-hook partition (see e.g. \cite[Equ.~(4.10)]{King}).

Suppose that  $\bgl=(\lambda^{(1)}; \ldots; \lambda^{(m)})$ is a multipartition of $n$. Then we define the supersymmetric Schur function associated with $\bgl$ by
\begin{eqnarray}\label{Equ:Supersymm-multi}
S_{\bgl}(\mathbf{x}/\mathbf{y})&=&\prod_{i=1}^{m}S_{{\lambda}^{(i)}}^{(i)}(\bx^{(i)}/\by^{(i)}),
\end{eqnarray}
which is a super analogue of  the Schur function associated to multipartitions defined in \cite[(6.1.2)]{S}.
\end{point}

\begin{point}{}* Now we give a combinatorial interpretation of the supersymmetric Schur functions in terms of $\bk|\bl$-semistandard tableau, which shows $S_{\bgl}(\bx/\by)=0$ unless $\bgl$ is $(\bk,\bl)$-hook multipartition of $n$.

Recall that for $\bgl\in \mpn$, the \textit{$\bk|\bl$-semistandard tableau} $\ft$ of shape $\bgl$ is a filling of boxes of $\bgl$ with variables  $\bx$, $\by$ such that for each $i$:\end{point}
\begin{enumerate}
  \item[(a)] The $i$-component $\ft^{(i)}$ of $\ft$ contains variables  $\bx^{(i)}, \by^{(i)}$ with shape $\lambda^{(i)}$, the $\bx$ part (the boxes filled with variables $\bx$ of $\ft^{(i)}$) is a tableau and the $\by$ part is a skew tableau;
  \item[(b)] The $\bx$ part is nondecreasing in rows, strictly increasing in columns;
  \item[(c)] The $\by$ part is nondecreasing in columns, strictly in creasing in rows.
\end{enumerate}
We denote by $\mathrm{std}_{\bk|\bl}(\bgl)$ the set  of $\bk|\bl$-semistandard tableaux of shape $\bgl$ and by  $s_{\bk|\bl}(\bgl)$ its cardinality. Clearly $\ft=(\ft^{(1)};\ft^{(2)}; \ldots; \ft^{(m)})$ is a $\bk|\bl$-semistandard tableau if and only if $\ft^{(i)}$ is a $(k_i,\ell_i)$-semistandard tableau in the sense of \cite[Definition~2.1]{B-Regev} for all $1\leq i\leq m$. Thanks to \cite[\S\,2]{B-Regev} and \cite[Lemma~4.2]{BKK}, $s_{\bk|\bl}(\bgl)\neq 0$ if and only if $\bgl\in H(\bk|\bl;m,n)$.

For a $\bk|\bl$-semistandard tableau $\ft$ of shape $\bgl$,  we define
\begin{equation*}
  \ft(\bx/\by)=\prod_{i=1}^{m}\prod_{(a,b)\in\lambda^{(i)}}z_{a,b}
\end{equation*}
where $z_{a,b}=x_{a,b}$ (resp. $-y_{a,b}$) whenever the box $(a,b)$ of $\ft^{(i)}$ is filled with symbol $x_{a,b}\in\bx$ (resp. $y_{a,b}\in \by$). It follows from \cite{B-Regev} and \cite[Equ.~(4.9)]{King} that for $\bgl\in H(\bk|\bl;m,n)$, we have
\begin{equation*}
  S_{\bgl}(\bx/\by)=\sum_{\ft\in\mathrm{std}_{\bk|\bl}(\bgl)}\ft(\bx/\by).
\end{equation*}

\begin{point}{}*\label{Point:Super-Hall} For positive integer $a$, let $q_{a}(\bx;t)$ be the Hall-Littlewood symmetric function for the variables $\bx=(x_1, \ldots, x_k)$ (cf. \cite[III, 2]{Macdonald}), which is defined by
\begin{eqnarray*}
  &&q_{0}(\bx;t)=1,\\
  &&q_a(\bx;t)=(1-t)\sum_{i=1}^kx_i^a\prod_{j\neq i}\frac{x_i-tx_j}{x_i-x_j}\qquad(a\geq 1).
\end{eqnarray*}

It is known that the generating function for the $q_a(\bx;t)$ is
\begin{equation*}
 Q(u)=\sum_{a\geq 0}q_a(\bx;t)u^{a}=\prod_{i=1}^k\frac{1-x_itu}{1-x_iu}.
\end{equation*}

Following \cite{M2010}, we define the \textit{super Hall-Liitlewood function} $q_a(\bx/\by;t)\in \Lambda_{k,\ell}(t)$ as follows:
\begin{equation*}
 Q_{\bx/\by}(u)=\sum_{a\geq 0}q_a(\bx/\by;t)u^{a}=\prod_{i=1}^k\frac{1-x_itu}{1-x_iu}
 \prod_{j=1}^{\ell}\frac{1-y_ju}{1-y_jtu}.
\end{equation*}

From the definition of $q_{a}(\bx/\by;t)$, we have (see \cite[(2.1)]{M2010})
\begin{eqnarray*}
  q_a(\bx/\by;t) &=&\sum_{i=0}^{a}q_{i}(\bx;t)q_{a-i}(t\by;t^{-1})\\
  &=&\sum_{i=0}^at^{a-i}q_i(\bx;t)q_{a-i}(\by;t^{-1}).
\end{eqnarray*}
\end{point}

\begin{point}{}*\label{subsec:standard-elements}Given a composition $\mathbf{c}=(c_1, \ldots, c_b)$ of $ n$, we denote by  $\mathfrak{S}_n^{(i)}$  the subgroup of $\mathfrak{S}_n$ generated by $s_j$ such that $a_{i-1}+1<j\leq a_i$, where $a_i=\sum_{j=1}^ic_j$ for $1\leq i\leq b$.  Then $\mathfrak{S}_n^{(i)}\backsimeq \mathfrak{S}_{c_i}$. Now we define a parabolic subgroup $\mathfrak{S}_{\mathbf{c}}$ of $\mathfrak{S}_n$ by
\begin{equation*}
 \mathfrak{S}_{\mathbf{c}}=\mathfrak{S}_n^{(1)}\times \mathfrak{S}_n^{(2)}\times\cdots\times \mathfrak{S}_n^{(b)}.
\end{equation*}In other words,  $\mathfrak{S}_{\mathbf{c}}$ is the Young subgroup of $\mathfrak{S}_n$ associated to $\mathbf{c}$.

Let $W_{\mathbf{c}}$ be the subgroup of $W_{m,n}$ obtained as the semidirect product of $\mathfrak{S}_{\mathbf{c}}$ with $(\mathbb{Z}/m\bz)^{n}$. Then $W_{\mathbf{c}}$ can be written as
\begin{equation*}
  W_{\mathbf{c}}=W^{(1)}\times W^{(2)}\times\cdots\times W^{(b)},
\end{equation*}
where $W^{(1)}$ is the subgroup of $W_{m,n}$ generated by $\mathfrak{S}_n^{(i)}$ and $t_j$ such that $a_{i-1}<j\leq a_i$. Then $W^{(i)}\simeq W_{m,c_i}$, which enables us to yield a natural embedding
\begin{equation*}
  \theta_{\mathbf{c}}:  W_{m,c_1}\times\cdots\times W_{m,c_b}\hookrightarrow  W_{m,n}
\end{equation*}
for each composition $\mathbf{c}$ of $n$.

For positive integer $i$, we define
\begin{eqnarray*}
  w(1,i)=t_1^i, & & w(a,i)=t_a^is_{a-1}\cdots s_1, \quad 2\leq a\leq n.
\end{eqnarray*}
Let $\mu=(\mu_1,\mu_2,\cdots)$ be a partition of n. We define
\begin{eqnarray*}
  w(\mu,i)&=&w(\mu_1,i)\times  w(\mu_2,i)\times\cdots
\end{eqnarray*}
with respect to the embedding $\theta_{\mu}$.  More generally, for $\bmu=(\mu^{(1)}, \ldots, \mu^{(m)})\in\mpn$,  we define
\begin{eqnarray*}
  w(\bmu) &=& w(\mu^{(1)},1)w(\mu^{(2)},2)\cdots w(\mu^{(m)},m).
\end{eqnarray*}
Then $\{w(\bmu)|\bmu\in\mpn\}$ is a set of conjugate classes representatives for $W_{m,n}$ (see \cite[4.2.8]{JK}).
\end{point}

\vspace{2\jot}
Following \cite[\S4]{King}, we define the  \textit{power sum supersymmetric function} $p_i(\bx/\by)$ to be $p_0(\bx/\by)=1$ and $p_i(\bx/\by)=p_i(\bx)-p_i(\by)$ for $i\geq 1$, where $p_i(\bx)$ denotes the $i$-th power sum symmetric function with respect to the variables $\bx$. For a partition $\gl=(\gl_1, \gl_2, \cdots, \gl_{r})$ of $n$, the power sum supersymmetric function $p_{\gl}(\bx/\by)\in \Lambda_{k,\ell}$ is defined as $p_{\gl}(\bx/\by)=\prod_{i=1}^rp_{\gl_i}(\bx/\by)$.

We will need the following fact, which clarifies the relationship between supersymmetric Schur functions and power sum supersymmetric functions indexed by partitions of $n$.

\begin{proposition}[\protect{\cite[(4.4)]{King}}]\label{Prop:King}Let $\chi^{\gl}$ be the irreducible character of $\mathbb{C}\mathfrak{S}_n$ corresponding to $\gl\vdash n$ and let $\chi^{\gl}(\mu)$ be the value of $\chi^{\gl}$ at elements of cycle type $\mu$. Then
\begin{equation*}
  s_{\gl}(\bx/\by)=\sum_{\mu\vdash n}Z_{\mu}^{-1}\chi^{\gl}(\mu)p_{\mu}(\bx/\by).
\end{equation*}
where  $Z_{\mu}$ is the order of the centralizer in $\mathfrak{S}_n$ of an element of cycle type $\mu$.
 \end{proposition}

Recall that $\varsigma$ is a fixed primitive $m$-th root of unity. Let $\bgl=(\gl^{(1)}, \ldots, \gl^{(m)})$ be a multipartition of $n$. Then we define the power sum supersymmetric function $P_{\bgl}(\bx/\by)$ associated to $\bgl$ as follows. For each integer $a\geq 1$ and $i$ such that $1\leq i\leq m$, set
\begin{equation*}
  P_a^{(i)}(\bx/\by)=\sum_{j=1}^m\varsigma^{-ij}p_{a}(\bx^{(j)}/\by^{(j)}).
\end{equation*}
For a partition $\gl^{(i)}=(\gl_1^{(i)}, \gl_2^{(i)}, \cdots, \gl_{r}^{(i)})$, we define a function $P_{\gl^{(i)}}(\bx/\by)\in \Lambda_{k,\ell}$ by
\begin{equation*}
  P_{\gl^{(i)}}(\bx/\by)=\prod_{j=1}^rP_{\gl_j^{(i)}}^{(i)}(\bx/\by),
\end{equation*}
and define $P_{\bgl}(\bx/\by)$ by
\begin{equation*}
  P_{\bgl}(\bx/\by)=\prod_{i=1}^mP_{\gl^{(i)}}(\bx/\by).
\end{equation*}

By abusing notations, we denote by $\chi^{\bgl}$ the irreducible character of $W_{m,n}$ corresponding to $\bgl\in\mpn$ and write $\chi^{\bgl}(\bmu)=\chi^{\bgl}(w(\bmu))$. Thanks to Proposition~\ref{Prop:King}, we can prove the following formula by applying the same argument as the one in
\cite[Appendix~B]{Macdonald}, which is a super-analogue of \cite[Appendix~B(9.5)]{Macdonald} and \cite[Proposition~6.6]{S}.
\begin{proposition}\label{Prop:super-Schur-power}Keep notations as above. Then
  \begin{equation*}
    S_{\bgl}(\bx/\by)=\sum_{\bmu\in\mpn}Z_{\bmu}^{-1}\chi^{\bgl}(\bmu)P_{\bmu}(\bx/\by),
  \end{equation*}
  where $Z_{\bmu}$ is the order of the centralizer in $W_{m,n}$ of an element of cycle type $\bmu$.
\end{proposition}

\section{The super Frobenius formula}
In this section, we define an operator $D$ on $V^{\otimes n}$ and compute the trace of product of $D$ and standard elements of  $\mathcal{H}$ by applying the technique initiated by Ram in \cite{Ram}, which enable us present the super Frobenius formula for the characters of $\mathcal{H}$.

\begin{point}{}*
Let $\mathbb{K}'=\mathbb{K}(z_1, z_2, \ldots, z_{k+\ell})$ be the field of rational functions on $\mathbb{K}$. In the remainder of this paper, we assume that $\mathcal{H}$, $U_q(\fg)$, $V$, etc., are defined over $\mathbb{K}'$ and use the same notations such as $\mathcal{H}$, $U_q(\fg)$, $V$, etc. Let us remark that Theorem~\ref{Them:Schur-Weyl} holds for $\mathbb{K}'$, and we may and will identify $h\in \mathcal{H}$ with $\Phi(h)\in \mathrm{End}_{U_q(\mathfrak{g})}(V^{\otimes n})$.

Let \begin{eqnarray*}
          &&\mathscr{I}^{+}(n;k|\ell)=\{\bi=(i_1, \ldots, i_n)|1\leq i_1\leq i_2\leq\cdots\leq i_r\leq k+\ell\},\\
      &&\mathscr{C}(n;k|\ell)=\{\bc=(c_1, \ldots, c_{k+\ell})|c_a\geq 0, |\bc|=\sum_{a=1}^{k+\ell} c_a=n\}.
    \end{eqnarray*}
For $\bi=(i_1, \ldots, i_n)\in \mathscr{I}(n;k|\ell)$, we define $\mathrm{wt}(\bi)=(c_1, \ldots, c_{k+\ell})$ with $c_a$ equaling  the times of $a$ appearing in $\bi$. Given $\bc\in\mathscr{C}(n;k|\ell)$, let $V_{\bc}^{\otimes n}$ be the subspace of $V^{\otimes n}$ defined by
\begin{equation*}
 V_{\bc}^{\otimes n}=\mathrm{Span}_{\mathbb{K}'}\left\{v_{\bi}\left|\bi\in\mathscr{I}(n;k|\ell)\text{ with }\mathrm{wt}(\bi)=\bc\right\}\right..
\end{equation*}
For $i=1, \ldots, k+\ell$, we define $P_i\in\mathrm{End}_{\mathbb{K}'}(V)$ to be the projections from $V$ onto $V_i$, that is, $P_i(v_k)=\delta_{ik}v_k$ and let
\begin{equation*}
  P_{\bc}=\sum_{\mathrm{wt}(\bi)=\bc}P_{i_1}\otimes\cdots\otimes P_{i_n}.
\end{equation*}
Then $P_{\bc}$ is a projection from $V^{\otimes n}$ to $V_{\bc}^{\otimes n}$. Let us remark that $P_{\bc}$ commutes with the action of $\mathcal{H}$, which means $P_{\bc}$ belongs to $\Psi^{\otimes n}(U_q(\fg))$.

For $\bc=(c_1, \ldots, c_{k+\ell})\in\mathscr{C}(n;k|\ell)$, we let $z^{\bc}=z_1^{c_1}\cdots z_{k+\ell}^{c_{k+\ell}}$ and define an operator $D$ on $V^{\otimes n}$ by
\begin{equation*}
  D=\sum_{\bc\in\mathscr{C}(n,k|\ell)}z^{\bc}P_{\bc}.
\end{equation*}
From now on, we replacing $z_1, \ldots, z_k$ by $x_1, \ldots, x_k$ and $z_{k+1}, \ldots, z_{k+\ell}$ by $-y_1, \ldots, -y_{\ell}$.

Let $\mathbb{C'}=\mathbb{C}(\bx,\by)$. In a similar way as above, we consider the $U(\fg)\otimes \mathbb{C}'W_{m,n}$-module $\widetilde{V}^{\otimes n}$, where $\widetilde{V}=\mathbb{C'}\otimes_{\mathbb{C}}\overline{V}$, and define a similar operator $D=\sum_{\bc\in\mathscr{C}(n,k|\ell)}z^{\bc}p_{\bc}$ over $\widetilde{V}^{\otimes n}$, where
$p_{\bc}$ is a projection from $\widetilde{V}^{\otimes n}$ to $\widetilde{V}_{\bc}^{\otimes n}$.

By a similar computation as the proof of \cite[Lemma~3.1]{Zhao}, we obtain the following fact.
\end{point}

\begin{lemma}\label{Lemm:Trace=Super-power}
  Let $w(n,i)=t_n^is_ns_{n-1}\cdots s_2$. Then
  \begin{equation*}
    \mathrm{Trace}(Dw(n,i),\widetilde{V}^{\otimes n})=P_n^{(i)}(\bx/\by).
  \end{equation*}
  More generally, for $\bgl\in \mpn$, we have
  \begin{equation*}
    \mathrm{Trace}(Dw(\bgl),\widetilde{V}^{\otimes n})=P_{\bgl}(\bx/\by).
  \end{equation*}
\end{lemma}

 \begin{lemma}\label{Lemm:Trace=Super-Schur} Let $p_{\bgl}$ be a primitive idempotent in $\mathbb{C}W_{m,n}$ affording the irreducible representation of $W_{m,n}$ corresponding to $\bgl\in\mpn$. Then
 \begin{equation*}
    \mathrm{Trace}(Dp_{\bgl},\overline{V}^{\otimes n})=S_{\bgl}(\bx/\by).
 \end{equation*}
 \end{lemma}
 \begin{proof}
   The proof can be done by applying Ram's argument in \cite[Lemma~3.7]{Ram}. In fact, Ram's argument shows
   \begin{equation*}
     \mathrm{Trace}(Dp_{\bgl}, \overline{V}^{\otimes n})=|W_{m,n}|^{-1}\sum_{w\in W_{m,n}}\chi^{\bgl}(w)\mathrm{Trace}(Dw, \overline{V}^{\otimes n}).
   \end{equation*}
   Since $D$ commutes with the action of $W_{m,n}$, $\mathrm{Trace}(Dw)$ only depends on the conjugate class containing $w$. Thus Lemma~\ref{Lemm:Trace=Super-power} implies
   \begin{equation*}
     \mathrm{Trace}(Dp_{\bgl}, \overline{V}^{\otimes n})=|W_{m,n}|^{-1}\sum_{w\in W_{m,n}}c_{\bmu}\chi^{\bgl}(\bmu)\mathrm{Trace}(Dw(\bmu), \overline{V}^{\otimes n}),
   \end{equation*}
   where $c_{\bmu}$ is the cardinality of the conjugate class of cycle type $\bmu$.
   Now the lemma follows by applying Proposition~\ref{Prop:super-Schur-power} and Lemma~\ref{Lemm:Trace=Super-power}.
 \end{proof}

 Now we can prove the computing formula for the trace of $Dh$ for all $h\in \mathcal{H}$.
\begin{proposition}\label{Prop:Trace-Dh}
  For any $h\in \mathcal{H}$, we have
  \begin{equation*}
    \mathrm{Trace}(Dh, V^{\otimes n})=\sum_{\bgl\in\mpn}\chi^{\bgl}(h)S_{\bgl}(\bx/\by).
  \end{equation*}
\end{proposition}
\begin{proof}The proof is done in the same way as in the proof of \cite[Proposition~6.6]{S}. By \cite[Theorem~5.14]{S}, there exists a complete family of mutually orthogonal primitive idempotents
\begin{equation*}
  \{P_{i}^{\bgl}|\bgl\in\mpn, i=1, \cdots, d_{\bgl}=\dim S^{\bgl}\}
\end{equation*}
 such that each $\phi(P_{i}^{\bgl})$ is a primitive idempotent corresponding to the irreducible representation of $\mathbb{C}W_{m,n}$ indexed by $\bgl$. By the same argument as in the proof of \cite[Theorem~4.4]{M2010}, we get
\begin{equation}\label{Equ:trace}
  \mathrm{Trace}(Dh, V^{\otimes n})=\sum_{\bgl\in\mpn}\sum_{i=1}^{d_{\bgl}}h_{ii}^{\bgl}\mathrm{Trace}(DP_i^{\bgl},V^{\otimes n}),
\end{equation}
where $h_{ii}^{\bgl}$ is the diagonal entry of the matrix of $h$ in $S^{\bgl}$.

 Using Ram's argument in the proof of \cite[Lemma~3.5]{Ram}, we can show $\mathrm{Trace}(DP_{i}^{\bgl}, V^{\otimes n})$ is independent of  $q$ and $\mathbf{Q}$. Therefore it is invariant under the specialization homomorphism $\phi$. As a consequence, we obtain
 \begin{equation*}
   \mathrm{Trace}(DP_i^{\bgl},V^{\otimes n})=\mathrm{Trace}(D\phi(P_i^{\bgl}),\overline{V}^{\otimes n})
   =S_{\bgl}(\bx/\by)
 \end{equation*}
 due to Lemma~\ref{Lemm:Trace=Super-Schur} and complete the proof.
\end{proof}

For $\bmu\in \mathscr{P}_{m,n}$, we define the \textit{standard element} $g(\bmu)\in\mathcal{H}$ of type $\bmu$ by imitating the definition of $w(\bmu)$ in \S\ref{subsec:standard-elements} as follows.
First for $a\geq 2$, we put $g(a,i)=\xi_a^ig_{a}\cdots g_2$.  Then for a partition $\mu=(\mu_1,\mu_2,\ldots)\vdash n$, we define
\begin{equation*}
  g(\mu,i)=g(\mu_1,i)\times g(\mu_2,i)\times\cdots
\end{equation*}
Finally, for $\bmu=(\mu^{(1)};\mu^{(2)}; \ldots; \mu^{(m)})\in\mpn$, we define $g(\bmu)\in\mathcal{H}$ by
\begin{eqnarray}\label{Equ:T-bmu}
   &&g(\bmu)=g(\mu^{(1)},1)\times g(\mu^{(2)},2)\times\cdots\times g(\mu^{(m)},m).
\end{eqnarray}
More generally, we define $g(w)=\xi_1^{c_1}\cdots\xi_n^{c_n}g_\sigma$ for $w=t_1^{_1}\cdots t_n^{c_n}\sigma\in W_{m,n}$ (see \S\ref{subsec:G(m,1,n)}). In \cite[Proposition~7.5]{S}, Shoji proved that the characters $\chi^{\bgl}$ of $\mathcal{H}$ are completely by their values on $g(\bmu)$ for all $\bmu\in\mpn$. For simplicity, we write $\chi^{\bgl}(\bmu)=\chi^{\bgl}(g(\bmu))$.

From now on, we will write $T(w)=\Phi(g(w))$ for $w\in W_{m,n}$, that is, $T(w)=\Omega^{c_1}_1\cdots \Omega_{n}^{c_n}T_{n}\cdots T_2$ for $w=t_1^{c_1}t_2^{_2}\cdots t_n^{c_n}s_n\cdots s_2\in W_{m,n}$.  Furthermore, we write $T(\bmu)=\Phi(g(\bmu))$.

Note that for any $\bi\in\mathscr{I}^{+}(n;k|\ell)$, $\bi$ may be written uniquely as the following form
\begin{align*}\label{Equ:i=(alpha;beta)}
 \bi(\alpha;\beta)&=(\stackrel{\alpha_1^{(1)}}{\overbrace{1,\ldots,1}}, \cdots,\stackrel{\alpha_{k_1}^{(1)}}{\overbrace{k_1,\ldots,k_1}};
 \stackrel{\beta_1^{(1)}}{\overbrace{k_1\!\!+\!\!1, \ldots, k_1\!\!+\!\!1}}, \cdots, \stackrel{\beta_{\ell}^{(1)}}{\overbrace{d_1,\ldots, d_1}};\cdots;\stackrel{\alpha_1^{(m)}}{\overbrace{d_{m\!-\!1}\!\!+\!\!1,\ldots,
 d_{m\!-\!1}\!\!+\!\!1}}, \\&\qquad\cdots,\stackrel{\alpha_{k_m}^{(m)}}{\overbrace{d_{m\!-\!1}\!\!+\!\!k_m,\ldots,d_{m\!-\!1}\!\!+\!\!k_m}};
 \stackrel{\beta_1^{(m)}}{\overbrace{d_m\!\!-\ell_m\!\!+\!\!1, \ldots, d_m\!\!-\ell_m\!\!+\!\!1}}, \cdots, \stackrel{\beta_{\ell_m}^{(m)}}{\overbrace{d_m,\ldots, d_m}})
\end{align*}for some $\alpha=(\alpha^{(1)};\ldots;\alpha^{(m)})$, $\beta=(\beta^{(1)};\ldots;\beta^{(m)})$ with $(\alpha;\beta)\in\mathscr{C}(n;k|\ell)$. Thus we may identify $\mathscr{I}^{+}(n;k|\ell)$ with $\mathscr{C}(n;k|\ell)$ as above. Further  we define the following function
\begin{equation*}
  \tilde{q}_{(\alpha;\beta)}(\bx/\by;q)=(\!-\!1)^{|\beta|-\ell(\beta)}q^{|\alpha|-\ell(\alpha)+\ell(\beta)-|\beta|}
   (q\!-\!q^{-\!1})^{\ell(\alpha;\beta)-1}\bx^{\alpha}(\!-\by)^{\beta},
\end{equation*}
for each $(\alpha;\beta)\in\mathscr{C}(n;k|\ell)$.

Let us remark that the relationship between $\tilde{q}_{(\alpha;\beta)}(\bx/\by;q)$ and super Hall-Littlewood function is described as following
(see \cite[Theorem~5.3]{Mit}):
\begin{equation}\label{Equ:q-q}
  \sum_{(\alpha;\beta)\in\mathscr{C}(n;k|\ell)}\tilde{q}_{(\alpha;\beta)}(\bx/\by;q)=
  \frac{q^n}{q-q^{-1}}q_n(\bx/\by;q^{-2}).
\end{equation}

Now we can determine the trace of $DT(w)$ for $w=t_1^{c_1}t_2^{_2}\cdots t_n^{c_n}s_n\cdots s_2\in W_{m,n}$.
\begin{lemma}\label{Lemm:Trace-Tw}
  Let $T(w)=\Omega^{c_1}_1\cdots \Omega_{n}^{c_n}T_{n}\cdots T_2$. Then
  \begin{equation*}
    \mathrm{Trace}(DT(w),V^{\otimes n})=\sum_{\substack{(\alpha;\beta)\in\mathscr{C}(n;k|\ell)}}
   \bq_{(\alpha;\beta)}^{\bc}\tilde{q}_{(\alpha;\beta)}(\bx/\by;q),
  \end{equation*}
  where $\bq_{(\alpha;\beta)}^{\bc}=Q^{c_1}_{c_1(\bi)}\cdots Q^{c_n}_{c_n(\bi)}$ with $\bi=(\alpha;\beta)$.
\end{lemma}
\begin{proof}
The computation of the action of $DT_{n}\cdots T_2$ is completely the same as the proof of \cite[Lemma~3.1]{Zhao}. Notice that in this computation, Equ.~(\ref{Equ:Ta-action}) implies that the contribution of the term $\bx^{*}(-\by)^{*}$ comes form the basis vector $\bi=(\alpha;\beta)$ with $(\alpha;\beta)\in\mathscr{C}(n;k|\ell)$ of $V^{\otimes n}$, and $\Omega_j(\bi)=Q_{c_j(\bi)}\bi$ for all $\bi\in \mathscr{I}^{+}(n;k|\ell)$. We complete the proof.
\end{proof}

For $\bc=(c_1, c_2, \ldots,c_m)\in\mathscr{C}(n;m)$. If $(\alpha;\beta)\in \mathscr{C}(n,k|\ell)$ satisfies $|\alpha^{(i)}|+|\beta^{(i)}|=c_i$ for $i=1,\ldots, m$, we write $(\alpha;\beta)\in\bc$.
The following formula expresses the trace of  $DT(n,i)$ on $V^{\otimes n}$ in terms of super Hall-Littlewood functions.

\begin{proposition}\label{Prop:Trace-DT(n,i)}Let $q_n^{(i)}(\bx/\by;q,\bq)=\mathrm{Trace}(DT(n,i), V^{\otimes n})$. Then
\begin{equation*}
  q_n^{(i)}(\bx/\by;q,\bq)=\frac{q^n}{q-q^{-1}}\sum_{\bc\in\mathscr{C}(n;m)}
  Q_{\bc}^i\prod_{j=1}^mq_{c_j}\left(\bx^{(j)}/\by^{(j)};q^{-2}\right).
\end{equation*}
where $Q_{\bc}$ denotes $Q_a$ with $a$ being the largest color of positive integers in $\mathbf{c}$.
\end{proposition}
\begin{proof} Notice that $Q_{(\alpha;\beta)}$ is independent for $(\alpha;\beta)\in \bc$ for fixed $\bc\in\mathscr{C}(n;m)$.   Thanks to Lemma~\ref{Lemm:Trace-Tw},
   \begin{eqnarray*}q_n^{(i)}(\bx,\by;q,\bq)&=&\sum_{\substack{\bc\in\mathscr{C}(n;m)}}Q^i_{\bc}
   \sum_{(\alpha;\beta)\in \bc}\tilde{q}_{(\alpha;\beta)}(\bx/\by;q)\\
   &=&(q-q^{-1})^{m-1}\sum_{\substack{\bc\in\mathscr{C}(n;m)}}Q^i_{\bc}
   \prod_{j=1}^m\sum_{(\alpha^{(j)};\beta^{(j)})\in \mathscr{C}(c_j,k_j|\ell_j)}\tilde{q}_{(\alpha^{(j)};\beta^{(j)}}(\bx^{(j)}/\by^{(j)};q)\\
   &=&(q-q^{-1})^{m-1}\sum_{\substack{\bc\in\mathscr{C}(n;m)}}Q^i_{\bc}
   \prod_{j=1}^m\left(\frac{q^{c_j}}{q-q^{-1}}q_{c_j}(\bx^{(j)}/\by^{(j)};q^{-2})\right)\\
   &=&\frac{q^{n}}{q-q^{-1}}\sum_{\substack{\bc\in\mathscr{C}(n;m)}}Q^i_{\bc}
   \prod_{j=1}^mq_{c_j}(\bx^{(j)}/\by^{(j)};q^{-2}),\end{eqnarray*}
   where the third equality follows from Equ.~(\ref{Equ:q-q}). It completes the proof.
\end{proof}

Let $q_n^{(i)}(\bx/\by;q,\bq)$ be as in Proposition~\ref{Prop:Trace-DT(n,i)}. For $\bmu=(\mu^{(1)};\mu^{(2)};\ldots;\mu^{(m)})\in\mpn$, we define a function $q_{\bmu}(\bx/\by;q,\bq)$ as follows:
\begin{equation}\label{Equ:q-bmu}
 q_{\bmu}(\bx/\by;q,\bq)=\prod_{i=1}^m\prod_{j=1}^{\ell(\mu^{(i)})}
 q^{(i)}_{\mu_j^{(i)}}(\bx/\by;q,\bq).
\end{equation}

Now we can obtain the super Frobenius formula for the characters of $\mathcal{H}$, which is a generalization of \cite[Theorem~6.14]{S} and \cite[Theorem~5.5]{M2010}.

\begin{theorem}\label{Them:Frobenius}
  For each $\bmu\in\mpn$,
  \begin{equation*}
    q_{\bmu}(\bx/\by;q,\bq)=\sum_{\bgl\in\mpn}\chi^{\bgl}(\bmu)S_{\bgl}(\bx/\by).
  \end{equation*}
\end{theorem}

\begin{proof}  For each $\bmu\in\mpn$, we have
\begin{eqnarray*}
 \sum_{\bgl\in\mpn}\chi^{\bgl}(\bmu)S_{\bgl}(\bx/\by) &=&\mathrm{Trace}(DT({\bmu}),V^{\otimes n})\\&=&\prod_{i=1}^m\prod_{j=1}^{\ell(\mu^{(i)})}\mathrm{Trace}\left(DT(\mu^{(i)},i),V^{\otimes \mu^{(i)}_j}\right)\\
 &=&\prod_{i=1}^m\prod_{j=1}^{\ell(\mu^{(i)})}
 q^{(i)}_{\mu_j^{(i)}}(\bx/\by;q,\bq)\\
 &=&q_{\bmu}(\bx/\by;q,\bq),
\end{eqnarray*}
where the first and the second equality follows by applying Proposition~\ref{Prop:Trace-Dh} and Eq.~(\ref{Equ:T-bmu}) respectively, the third one follows by Proposition~\ref{Prop:Trace-DT(n,i)}.
\end{proof}

We close the paper with the following remark which was suggested to the author by the referee.

\begin{remark}\label{Remark:Super-Frobenius} By applying the specialization homomorphism $\phi$ and Theorem~\ref{Them:Frobenius}, one obtains the super Frobenius formula for the complex reflection group $W_{m,n}$, that is, for each $\bgl, \bmu\in\mpn$, let $\chi^{\bgl}$ be the irreducible character of the irreducible representation of $W_{m,n}$ corresponding to $\bgl$ and write $\chi^{\bgl}(\bmu)=\chi^{\bgl}(w(\bmu))$. Then
  \begin{equation*}
    q_{\bmu}(\bx/\by;1,\boldsymbol{\varsigma})=\sum_{\bgl\in\mpn}\chi^{\bgl}(\bmu)S_{\bgl}(\bx/\by),
  \end{equation*}
where $\boldsymbol{\varsigma}=(\varsigma,\varsigma^2,\ldots, \varsigma^m=1)$.
\end{remark}
%%%%%%%%%%%%%%%%%%%%%%%%%%%%%%%%%%%%%%%%%%%%%%%%%%%%%%%%%%%%%%%%%%%%%%
%\bibliographystyle{amsplain}

\end{document}